\documentclass[12pt]{amsart}

\usepackage{amssymb}
\usepackage{amsthm}
\usepackage{amsmath}
\usepackage{fontenc}
\usepackage{inputenc}
\usepackage{enumitem}
\usepackage{hyperref}
\usepackage{times}
\usepackage{graphicx}
\usepackage{multirow}
\usepackage{xcolor}
\usepackage{colortbl}
\usepackage{color}
\usepackage{array}
\usepackage{wasysym}
\usepackage{tikz}
\usetikzlibrary{scopes,arrows,decorations.pathmorphing,backgrounds,positioning,fit,petri,shapes,calc}

\makeatletter
\@namedef{subjclassname@2010}{%

\textup{2010} Mathematics Subject Classification}

\makeatother

\def\real{\hbox{\rm\setbox1=\hbox{I}\copy1\kern-.45\wd1 R}}
\def\natural{\hbox{\rm\setbox1=\hbox{I}\copy1\kern-.45\wd1 N}}

\newtheorem{theorem}{Theorem}[section] 
\newtheorem{lemma}[theorem]{Lemma}     
\newtheorem{corollary}[theorem]{Corollary}
\newtheorem{proposition}[theorem]{Proposition}
\newtheorem{case}[theorem]{Case}
\newtheorem{question}[theorem]{Question}

\newtheorem{definition}[theorem]{Definition}

\title{Over Recurrence for Mixing Transformations}

\author[T. M. Adams]{Terrence M. Adams}

\email{terry@ganita.org}

\date{\today}

\begin{document}

\begin{abstract} 
We show that every invertible strong mixing transformation on a Lebesgue space 
has strictly over-recurrent sets.  Also, we give an explicit procedure for constructing 
strong mixing transformations with no under-recurrent sets. 
This answers both parts of the first question posed in \cite{Berg96}.  

We define $\epsilon$-over-recurrence and show that given $\epsilon > 0$,
any ergodic measure preserving invertible transformation (including discrete spectrum) 
has $\epsilon$-over-recurrent sets of arbitrarily small measure. 
Discrete spectrum transformations and rotations do not have over-recurrent sets, 
but we construct a weak mixing rigid transformation with strictly over-recurrent sets.
\end{abstract}


\maketitle

\section{Introduction}
We answer a two-part question posed in \cite{Berg96}.  It is the first question raised 
in \cite{Berg96_2} on page 50. 

\begin{question}
\label{Q1}
``Is it true that for any invertible mixing measure 

preserving system 
$(X, \mathcal{B}, \mu, T)$ there exists $A \in \mathcal{B}$ with 

$\mu (A) > 0$ 
such that for all $n\neq 0$, $\mu (A\cap T^n A) < \mu (A)^2$?  

How about the 
reverse inequality $\mu (A\cap T^n A) > \mu (A)^2$''
\end{question}

\noindent 
The answer is different for each part.  Respectively, the answers are "no", 
and "yes".  One of the key differences is a basic 
lemma on set intersections which is presented in Lemma \ref{lem1}.  
However, this lemma alone is not sufficient, and in particular, 
this lemma is pointed out in \cite{Berg96_2}.  
In the next section, we prove that the answer to the second part 
is "yes".  This is done by constructing a set of positive measure 
such that the given mixing transformation mixes the set slowly.  
This answers the same question raised in \cite{BFW16}. 

The notions of over-recurrent, under-recurrent, strictly over-recurrent and 
strictly under-recurrent are defined in \cite{BFW16} as a means 
for addressing Question \ref{Q1} and similar questions.  
We expand these definitions to include the weaker notion 
of $\epsilon$-over-recurrent and $\epsilon$-under-recurrent.   
See section \ref{prelim} for definitions.  
It is straight-forward to show 
that any partially rigid transformation has no under-recurrent set.  
The class of partially rigid transformations is larger than the class of rigid transformations 
and was first introduced by N. Friedman in \cite{Fri89}.  
As a preliminary result, we give a short proof that any discrete spectrum transformation (or rotation) 
does not have an over-recurrent set.  
However, we show that given an invertible ergodic measure preserving transformation $T$ and 
$\epsilon > 0$, $T$ has $\epsilon$-over-recurrent sets with arbitrarily small measure.  
Also, we construct a rigid weak mixing transformation that has 
a strictly over-recurrent set.  The question 
of whether every weak mixing transformation has an over-recurrent set remains open. 

To answer the first part of Question \ref{Q1}, we give a general 
procedure for constructing a strong mixing transformation 
from an input mixing transformation and an arbitrary rigid transformation.  
We gradually diminish the effects of the rigid transformation, but 
in the process, build a strong mixing transformation that acts 
like a rigid transformation on a shrinking part of the measure space. 
We use a technique from \cite{Towerplex1} to produce this 
slow mixing transformation. 

Finally, in the last section, we point out that the same construction 
for producing a slow mixing transformation can be used to construct 
a strong mixing transformation with singular spectrum 
from any strong mixing transformation.  
Thus, any strong mixing transformation can be multiplexed 
with any rigid transformation to produce a transformation that 
is mixing of all orders. 

\section{Preliminaries}
\label{prelim}
All transformations are assumed to be invertible, ergodic and measure preserving 
on a fixed Lebesgue probability space $(X, \mathcal{B}, \mu )$, 
and all sets are assumed to be measurable. 
Let $\natural = \{1, 2, \ldots \}$ be the natural numbers 
and $\mathbb{Z}$ the set of integers. 
The following definitions are expanded from \cite{BFW16}.  
\begin{definition}
Let $A$ be a measurable set such that $0 < \mu (A) < 1$. 
\begin{enumerate}
\item Set $A$ is over-recurrent if 
$\mu (T^nA\cap A) \geq \mu (A)^2$ for $n\in \mathbb{Z}$.
\item Set $A$ is under-recurrent if 
$\mu (T^nA\cap A) \leq \mu (A)^2$ for $n\in \natural$.
\item Set $A$ is strictly over-recurrent if 
$\mu (T^nA\cap A) > \mu (A)^2$ for $n\in \mathbb{Z}$.
\item Set $A$ is strictly under-recurrent if 
$\mu (T^nA\cap A) < \mu (A)^2$ for $n\in \natural$.
\item Set $A$ is $\epsilon$-over-recurrent if 
$\mu (T^nA\cap A) > (1 - \epsilon )\mu (A)^2$ for $n\in \mathbb{Z}$.
\item Set $A$ is $\epsilon$-under-recurrent if 
$\mu (T^nA\cap A) < (1 + \epsilon )\mu (A)^2$ for $n\in \natural$.
\end{enumerate}
\end{definition}
These definitions are motivated by the Khintchine recurrence theorem \cite{Khi34}. 
It was shown in \cite{BFW16} that there exist mixing transformations 
with no under-recurrent sets.  Also, under-recurrent functions are defined, 
and it is shown that any transformation with singular maximal spectral type 
has no under-recurrent function.  While we give a general construction 
of mixing transformations with no under-recurrent sets, our main results 
concern (strictly) over-recurrent sets.  
All results were obtained independently 
of \cite{BFW16}.  

\section{Over-recurrent Sets}
This section focuses on results related to (strictly) over-recurrent sets.  
First, we prove that any strong mixing transformation has a strictly 
over-recurrent set. 
\begin{theorem}
\label{mix-over-rec}
Let $T$ be an invertible mixing transformation on a Lebesgue probability space. 
Then $T$ has strictly over-recurrent sets $A$ of arbitrarily small measure.  
In particular, $\mu (T^nA\cap A) > \mu (A)^2$ for all $n \in \mathbb{Z}$.  
\end{theorem}

We use the following lemma in the construction of the over-recurrent sets. 
For completeness, a proof is included. 
\begin{lemma}
\label{seq1}
Let $a > 0$ be a real number.  
Define $a_i = \frac{a}{i(i+1)}$ for $i \in \natural$. 
Then for $k \in \natural$, 
\[
( \sum_{i=1}^{k} a_i  )^2 + \sum_{i=k+2}^{\infty} a_i  > 
a^2 \Big( 1 + \frac{k(1-2a) + (1 - 3a)}{a(k+1)(k+2)} \Big) . 
\]
\end{lemma}

\noindent {\bf Proof:}
First, note that 
\[
\sum_{i=1}^{\infty} a_i = a\sum_{i=1}^{\infty} ( \frac{1}{i} - \frac{1}{i+1} ) = a . 
\]
Thus, 
\begin{eqnarray*}
( \sum_{i=1}^{k} a_i  )^2 + ( \sum_{i=k+2}^{\infty} a_i ) &=& 
(a - \frac{a}{k+1})^2 + \frac{a}{k+2} \\ 
&=& a^2 - \frac{2a^2}{k+1} + \frac{a^2}{(k+1)^2} + \frac{a}{k+2} \\ 
&>& a^2 - \frac{2a^2}{k+1} + \frac{a^2}{(k+1)(k+2)} + \frac{a}{k+2} \\ 
&=& a^2 + \frac{ - 2a^2 (k+2) + a^2 + a(k+1) }{ (k+1)(k+2) } \\ 
&=& a^2 + \frac{ k ( a - 2a^2 ) + (a - 3a^2) }{ (k+1)(k+2) } \\ 
&=& a^2 \Big( 1 + \frac{ k ( 1 - 2a ) + (1 - 3a) }{ a(k+1)(k+2) } \Big) . \Box 
\end{eqnarray*}

\noindent {\bf Proof of Theorem \ref{mix-over-rec}:} 
Let $a\in \real$ be such that $0 < a < \frac{1}{3}$ and let $a_i = \frac{a}{i(i+1)}$ 
for $i\in \natural$.  
For $j \in \natural$, choose $\epsilon_j > 0$ and decreasing such that 
\begin{eqnarray}
\label{eps1} (1 - \epsilon_j)\bigg( 1 + \frac{j(1 - 2a) + (1 - 3a)}{a(j+1)(j+2)} \bigg) > 1 . 
\end{eqnarray}
We will define an infinite sequence $A_i$ of disjoint measurable sets such that 
$\mu (A_i) = a_i$ for $i\in \natural$, and $A=\bigcup_{i=1}^{\infty} A_i$.  
Let $A_1$ be any set with measure ${a} / {2}$. 
Since $T$ is mixing, there exists $N_1 \in \natural$ such that for $| n | \geq N_1$, 
\[
| \mu (T^n A_1 \cap A_1) - \mu (A_1)^2 | < \epsilon_1 \mu (A_1)^2 . 
\]
Choose $m_1 \in \natural$ such that 
\[
m_1 > \max{ \{ \frac{1}{\epsilon_1 a_2}, N_1 \} } . 
\]
Let $B_1$ be the base of a Rohklin tower of height $m_1^2$ such that 
\[
\mu ( \bigcup_{i=0}^{m_1^2 - 1} T^i B_1 ) > 1 - \epsilon_1 . 
\]
Choose a subset $I_1 \subset B_1$ such that the set 
\[
A_2 = \{ T^i x : x\in I_1, 0\leq i < m_1^2, T^i x \notin A_1 \} 
\]
satisfies $\mu (A_2) = a_2$.  
Note that for $i \in \natural$ such that $|i| < N_1$, 
\[
\mu (T^i A_2 \cap A) > \mu (A_2) - \frac{1}{m_1} > \mu(A_2) - \epsilon_1 a_2 = (1 - \epsilon_1) \mu(A_2) . 
\]
See the appendix for a visual representation of $A_1$ and $A_2$. 

We repeat this inductively.  Given $A_1, A_2, \ldots A_k$, choose 
$A_{k+1}$ in the following manner.  
Let $C_k = \cup_{i=1}^{k} A_i$.  
Since $T$ is mixing, there exists $N_k > N_{k-1}$ such that for $| n | \geq N_k$, 
\[
| \mu (T^n C_k \cap C_k) - \mu (C_k)^2 | < \epsilon_k \mu (C_k)^2 . 
\]
Choose $m_k \in \natural$ such that 
\[
m_k > \max{ \{ \frac{1}{\epsilon_k a_{k+1}}, N_k \} } .
\]
Let $B_k$ be the base of a Rohklin tower of height $m_k^2$ such that 
\[
\mu ( \bigcup_{i=0}^{m_k^2 - 1} T^i (B_k) ) > 1 - \epsilon_k . 
\]
Choose a subset $I_k \subset B_k$ such that the set 
\[
A_{k+1} = \{ T^i x : x\in I_k, 0\leq i < m_k^2, T^i x \notin C_k \} 
\]
satisfies $\mu (A_{k+1}) = a_{k+1}$.  
The set $A_{k+1}$ has the property that for $i \in \natural$ such that $|i| < N_k$, 
\[
\mu (T^i A_{k+1} \cap A) > (1 - \epsilon_k) \mu (A_{k+1}) . 
\]

Now, we show the set $A = \bigcup_{i=1}^{\infty} A_i$ is over-recurrent.  
If $n$ is a natural number such that $|n| < N_1$, then 
\begin{eqnarray*}
\mu (T^n (\bigcup_{k=2}^{\infty}A_k) \cap A) &>& (1 - \epsilon_1) \mu (\bigcup_{k=2}^{\infty} A_k) 
= \frac{1}{2}(1 - \epsilon_1)a \\ 
&>& \frac{1}{2} (1 - \frac{2 - 6a}{6}) a = (\frac{1}{3} + \frac{a}{2}) a > a^2 .
\end{eqnarray*}

Let $k\in \natural$, and $n\in \natural$ be such that $N_k \leq n < N_{k+1}$.  
The set $A$ is a disjoint union of the following three sets: 
$C_k, A_{k+1}, \bigcup_{i=k+2}^{\infty} A_i$.  
For convenience, set $C_{k,1} = C_k$, $C_{k,2} = A_{k+1}$ 
and $C_{k,3}=\bigcup_{i=k+2}^{\infty}A_i$. 
Thus,
\begin{eqnarray*}
\mu (T^n A \cap A) &=& 
\sum_{i=1}^{3} \mu (T^n C_{k,1} \cap C_{k,i}) + \sum_{i=2}^{3} \mu (T^n C_{k,i} \cap A) \\ 
&>& (1 - \epsilon_k) \mu (C_{k,1})^2 + (1 - \epsilon_{k+1}) \mu (C_{k,3}) \\ 
&\geq& (1 - \epsilon_k) \bigg( (a - \frac{a}{k+1})^2 + \frac{a}{k+2} \bigg) \\ 
&>& (1 - \epsilon_k) a^2 \bigg( 1 + \frac{k(1-2a) + 1 - 3a}{a(k+1)(k+2)} \bigg) \mbox{, by Lemma \ref{seq1},} \\ 
&>& a^2 ,\ \ \ \mbox{by ( \ref{eps1} ) } .\ \ \ \ \Box
\end{eqnarray*}

Note, the method used for choosing $A_i$ can be used to show that mixing 
transformations have no uniform rate over all measurable sets.  
Given any sequence $\delta_i \to 0$, there exist parameters $\epsilon_i, N_i, m_i$ 
and $A_i$ such that 
\[
\lim_{n\to \infty} \frac{\delta_n}{ \mu (T^n A \cap A) - \mu (A)^2 } = 0 . 
\] 
This is already well known to be true, and follows from a general argument 
of Krengel \cite{Kre78} on the lack of uniform rates for the ergodic theorem. 

\subsection{Over-recurrence for non-mixing transformations}
The previous result can be used to construct a rigid weak mixing transformation 
that has a strictly over-recurrent set.  First, we show that any discrete spectrum 
transformation does not have an over-recurrent set. 

\begin{proposition}
If $T$ has discrete spectrum, then $T$ has no over-recurrent sets. 
\end{proposition}

\noindent {\bf Proof:} 
Let $A$ be any measurable set such that $0 < \mu (A) < 1$.  
Since $T$ has discrete spectrum, there exist a sequence of 
refining towers of heights $h_n$ and integers $k_n \geq 2$ 
such that $h_{n+1} = k_n h_n$. 
Choose $m$ such that the tower of height $h_m$ has a union $J$ 
of levels that approximates $A$. 
In particular, choose $\delta$ and $m$ such that 
$2\delta < 1 - \mu (A)$ and 
\[
\mu (A\triangle J) < \frac{\delta \mu (A)}{4} . 
\]
Note that $\mu (T^{ih_m} J \cap J) = \mu (J)$ for all $i \in \mathbb{Z}$. 
Thus, $\mu (T^{ih_m} A \cap A) > (1 - \delta ) \mu (A) > \mu (A)^2$ for all $i \in \mathbb{Z}$.  
By the $L_2$ ergodic theorem, 
\[
\lim_{n\to \infty} \frac{1}{h_n} \sum_{i=0}^{h_n - 1} \mu (T^i A\cap A) = \mu (A)^2 . 
\]
Since $\{ ih_m | i\in \natural \}$ forms a subsequence of positive density in $\natural$, 
then there must exist $i$ and $j$ such that 
$0 < j < h_m$ and 
\[
\mu (T^{ih_m + j} A \cap A) < \mu (A)^2 . \ \ \Box 
\]
A similar argument can be used to show that ergodic rotations do not have over-recurrent sets. 

The following may seem a bit surprising intuitively, but it is not difficult to prove. 
\begin{proposition}
\label{erg-over-rec}
Given any invertible ergodic measure preserving transformation $T$ and $\epsilon > 0$, 
$T$ has $\epsilon$-over-recurrent sets of arbitrarily small measure.
\end{proposition}

\noindent {\bf Proof:} 
Let $\epsilon, \delta > 0$.  Let $S$ be a rank-one strong mixing transformation.  
By Theorem \ref{mix-over-rec}, $S$ has a strictly over-recurrent set $A$ 
with measure less than $\delta$.  
Choose a tower for $S$ of height $h$ and a union $J$ of levels from the tower 
that approximate $A$ well, and such that the complement of the tower 
has measure less than ${\epsilon \mu (A)}/{4}$.  
Also, assume 
\[
\mu (A \triangle J) < \frac{\epsilon}{4} \mu (A) . 
\]
Choose a Rokhlin tower for $T$ of height $h$ such that the complement of the tower 
has measure less than ${\epsilon \mu (A)}/{4h}$.  
There is a one-to-one onto correspondence between the levels of the $T$ tower 
and the levels of the $S$ tower.  Take the correspondence that preserves the order 
of the levels from top to bottom of the towers. 
The set $J$ in the $S$ tower matches a set $K$ in the $T$ tower. 
It is not difficult to prove that the set $K$ is $\epsilon$-over-recurrent 
for the transformation $T$. 
$\ \Box$

\subsection{Rigid weak mixing transformations with strictly over-recurrent sets}
We prove there exist rigid weak mixing transformations with strictly over-recurrent sets.  
\begin{theorem}
There exist rigid weak mixing transformations $T$ and sets $A$ such that 
for all $n \in \mathbb{Z}$, 
\[
\mu (T^n A\cap A) > \mu (A)^2 . 
\]
\end{theorem}

\noindent {\bf Proof:} 
Let $S_1$ be a rank-one mixing transformation such as Ornstein's mixing rank-one, 
or the (Adams-Smorodinsky) staircase transformation.  
By Theorem \ref{mix-over-rec}, there is an over-recurrent set $A_1$ of arbitrarily 
small measure. 
By the technique used in Proposition \ref{erg-over-rec}, there exists a tower 
of height $h$ such that the set $A$ is ${\epsilon}/{4}$-over-recurrent, 
even if we modify $S$ to be discrete spectrum from this point on. 
Similarly, we can cut this tower into $r_1$ subcolumns of equal width, 
and stack to produce a rigid time (as $r_n \to \infty$).  
Resume the definition of the mixing transformation $S_2$ similar to $S_1$.  
Then define a set $A_2$ as in Theorem \ref{mix-over-rec}, such that 
iterates of $A_2$ overlap itself for a long time (forward and backward in time).  
Since $S_2$ is mixing, it will mix $A_1 \cup A_2$ over time.  Once this happens 
sufficiently well, then introduce another rigid time $r_2$.  It will not 
disturb the near over-recurrence of $A_1 \cup A_2$.  The error in the near over-recurrence 
can be forced to be much smaller than the size of set $A_3$.  The set $A_3$ is 
defined to be nearly fixed for a long period of time compared to the last 
mixing times chosen for $S_2$, as it operates on $A_1 \cup A_2$.  

This is repeated inductively to produce a rigid weak mixing transformation.  
The arguments used in Theorem \ref{mix-over-rec} and Propostion \ref{erg-over-rec} 
can be applied here to show that the set 
$A = \bigcup_{i=1}^{\infty} A_i$ is strictly over-recurrent for the resulting 
transformation $T$.  The transformation $T$ may be defined as 
\[
T = \lim_{n\to \infty} S_n . 
\]
Although, each $S_n$ may be strong mixing, the limiting transformation 
$T$ will be rigid weak mixing, if $r_n \to \infty$. 
$\ \Box$

\section{Slow Strong Mixing Transformations} 
In this section, we prove the following theorem. 
\begin{theorem}
There exists a strong mixing transformation $T$ such that 
for every set $A$ satisfying $0 < \mu (A) < 1$, the following set is infinite: 
\[
\{ n \in \natural : \mu (T^nA\cap A) - \mu (A)^2 > 0 \} . 
\]
\end{theorem}
We use a technique from \cite{Towerplex1} to construct our example.  
In \cite{Towerplex1}, a method is given for combining two transformations 
to produce a third "multiplexed" transformation. 
In that paper, the two input transformations are 
a rigid ergodic transformation and a weak mixing transformation.  
The output transformation is a rigid weak mixing transformation.  
In this case, our input transformations are a strong mixing 
transformation and a rigid transformation.  The output is a strong mixing 
transformation.  

We use the following standard result from measure theory. 
\begin{lemma} 
\label{lem1}
Let $(X,\mu)$ be a probability space.  
Given $\epsilon > 0$ and $0<\alpha \leq 1$, there exists $N$ such that 
for any measurable sets $A_1, A_2, \ldots , A_N$ satisfying 
$\mu (A_i) = \alpha$ for $1\leq i\leq N$, there exist 
$1\leq j < k \leq N$ such that 
\[
\mu (A_j \cap A_k) > \alpha^2 - \epsilon . 
\]
\end{lemma}

\noindent {\bf Proof:} 
\begin{eqnarray}
\int ( \sum_{i=1}^{N} I_{A_i} )^2 d\mu 
&=& \sum_{i \neq j} \mu (A_i \cap A_j) + \sum_{i=1}^{N} \mu (A_i) \\ 
&\geq& ( \sum_{i=1}^{N} \mu (A_i) )^2 = N^2 \alpha^2 
\end{eqnarray}
Therefore, 
\begin{eqnarray}
\frac{1}{N^2} \sum_{i \neq j} \mu (A_i \cap A_j) &\geq& 
\alpha^2 - \frac{\alpha}{N} 
\end{eqnarray}
and we have our result. 
$\ \Box$

\subsection{Mixing Counterexample}
The towerplex method was first defined in section 2 of \cite{Towerplex1}.  
The roles of the input transformations are different.  
In this case, we use $S$ to represent the first input transformation 
which will be a strongly mixing transformation.  
The second input transformation will be a rigid transformation denoted by $R$.  
Thus, a sequence of transformations $S_n:Y_n \to Y_n$ will be defined 
such that $S_n$ is isomorphic to $S$, and another sequence 
$R_n:X_n \to X_n$ such that $R_n$ is isomorphic to $R$.  
For each $n\in \natural$, $X_n \cup Y_n = X$ and define 
\begin{eqnarray*} 
T_{n}(x)= 
\left\{\begin{array}{ll}
R_{n}(x) & \mbox{if $x\in X_{n}$} \\ 
S_{n}(x) & \mbox{if $x\in Y_{n}$} . 
\end{array}
\right.
\end{eqnarray*}
Then the output transformation is defined by $T(x) = \lim_{n\to \infty} T_n(x)$ for $x \in X$. 

Two main parameters are used to control the properties of $T$:  
\begin{eqnarray}
s_n = \frac{1}{2(n+2)} &\mbox{and}& r_n = \frac{1}{2} . 
\end{eqnarray} 
The parameter $s_n$ represents the proportion of mass that transfers 
from $Y_n$ to $X_n$ at each stage.  Similarly, $r_n$ represents the proportion 
of mass that transfers from $X_n$ to $Y_n$ at each stage.  
These settings cause $\lim_{n\to \infty} \mu (Y_n) = 1$ and 
consequently $\lim_{n\to \infty} \mu (X_n) = 0$.  
Note, the fact that $S_n$ is mixing is not sufficient to prove that 
$T$ is mixing.  On the other hand, the fact that $T_n$ is not ergodic 
does not prevent $T$ from being ergodic.  We are more careful 
about defining $S_{n+1}$ based on $S_n$ and use 
a property called "isomorphism chain consistency" 
to show that $T$ is strongly mixing. 

This provides a general technique for constructing a slow strong mixing 
transformation from an arbitrary strong mixing transformation and 
an arbitrary rigid transformation. 
As in \cite{Towerplex1}, we can have for each $n \in \natural$, 
${1} / {(n+2)} < \mu (X_n) < {1} / {n}$ and 
${(n+1)} / {(n+2)} > \mu (Y_n) > {(n-1)} / {n}$.  
Also, since the measure of $X_n$ goes to zero slow enough, 
and the $X_n$ are approximately independent, then 
$X_n$ will mix with any measurable set.  
Also, $R_n$ rigid on $X_n$ will cause $T_n$ to be 
approximately ${1} / {(n+1)}$ rigid on $X$.  
In this way, we can slow the rate of mixing, because $T$ 
will resemble a rigid transformation for arbitrarily long 
times on $X_n$.  

\subsection{Slow mixing from dissipating rigidity}
Lemma \ref{lem1} will inform us on how long $S_n$ should 
run before phasing in $S_{n+1}$ to guarantee the intersection 
\begin{eqnarray}
\label{eqn1}
\mu (S_n^i A \cap A) > ( 1 - \frac{\delta}{n} ) \mu (A)^2 
\end{eqnarray}
for some $i$ and $\delta$.  Let $N$ be large enough to guarantee 
(\ref{eqn1}) holds for some $i$ from a subset of at least $N$ iterates. 
Let $\rho_k$ be a rigidity sequence for $R_n$.  
Choose $\rho_1, \rho_2, \ldots , \rho_N$ such that 
\begin{eqnarray}
\label{eqn2}
\mu (R_n^{\rho_i} A \cap A) > (1 - \frac{\delta}{n}) \mu (A) 
\end{eqnarray}
for $A \in P_n$.  
Using a similar approximation as in \cite{Towerplex1}, 
then a rigidity condition like (\ref{eqn2}) extends to all 
measurable sets $A$ such that $0 < \mu (A) < 1$.  
Thus, (\ref{eqn1}) and (\ref{eqn2}) together can be used to show 
that the following set is infinite: 
\[
\{ n\in \natural : \mu (T^i A\cap A) - \mu (A)^2 > 0 \} . 
\]

\section{Mixing Towerplex Details}
Partition $X$ into two equal 
sets $X_1$ and $Y_1$ (i.e. $\mu (X_1)=\mu (Y_1)=1/2$). 
Initialize $R_1$ isomorphic to $R$ and $S_1$ isomorphic to $S$ to operate on $X_1$ and $Y_1$, respectively. Define $T_1(x)=R_1(x)$ for $x\in X_1$ and $T_1(x)=S_1(x)$ for $x\in Y_1$.  Produce Rohklin towers of height $h_1$ with residual less than $\epsilon_1/2$ for each of $R_1$ and $S_1$. In particular, let $I_1, J_1$ be the base of the $R_1$-tower and $S_1$-tower such that $\mu (\bigcup_{i=0}^{h_1-1}R_1^iI_1)>1/2(1-\epsilon_1)$ and 
$\mu (\bigcup_{i=0}^{h_1-1}S_1^iJ_1)>1/2(1-\epsilon_1)$.  
Let 
$X_1^*=X_1\setminus \bigcup_{i=0}^{h_1-1}R_1^i(I_1)$ 
and 
$Y_1^*=Y_1\setminus \bigcup_{i=0}^{h_1-1}S_1^i(J_1)$ be the 
residuals for the $R_1$ and $S_1$ towers, respectively. 
Choose $I^{\prime}_1\subset I_1$ and $J^{\prime}_1\subset J_1$ such that 
$$
\mu (I^{\prime}_1)=r_1\mu (I_1)\mbox{ and }\mu (J^{\prime}_1)=s_1\mu (J_1).
$$
Set $X_2=X_1\setminus [\bigcup_{i=0}^{h_1-1}R_1^i(I_1^{\prime})]\cup [\bigcup_{i=0}^{h_1-1}S_1^i(J_1^{\prime})]$ and 
$Y_2=Y_1\setminus [\bigcup_{i=0}^{h_1-1}S_1^i(J_1^{\prime})]\cup [\bigcup_{i=0}^{h_1-1}R_1^i(I_1^{\prime})]$. 
We will define second stage transformations $R_2:X_2\to X_2$ and $S_2:Y_2\to Y_2$. 
First, it may be necessary to add or subtract measure from
 the residuals so that $X_2$ is scaled properly to define $R_2$, and 
 $Y_2$ is scaled properly to define $S_2$. 

\subsection{Tower Rescaling}
In the case where $\mu (I_1^{\prime})\neq \mu (J_1^{\prime})$, we give 
a procedure for transferring measure between the towers and the 
residuals. This is done in order to consistently define $R_2$ and $S_2$ 
on the new inflated or deflated towers. 
Let $a=\mu (\bigcup_{i=0}^{h_1-1}R_1^iI_1)$ and 
$b=h_1(\mu (J^{\prime}_1)-\mu (I^{\prime}_1))$. 
Let $c$ be the scaling factor and $d$ the amount of measure 
transferred between $\bigcup_{i=0}^{h_1-1}S_1^i(J^{\prime}_1)$ 
and $X_1^*$. Thus, $a+b-d=ca$ and $1/2-a+d=c(1/2 - a)$. The goal is 
to solve two unknowns $d$ and $c$ in terms of the other values. 
Hence, $d=(1-2a)b$ and $c=1+2b$. 

\subsubsection{$R$ Rescaling}
If $d>0$, define $I_1^*\subset J^{\prime}_1$ such that 
$\mu (I_1^*)=d/h_1$. Let 
$X_1^{\prime}=X_1^*\cup (\bigcup_{i=0}^{h_1-1}R_1^i(I_1^*))$. 
If $d=0$, set $X_1^{\prime}=X_1^*$.  
If $d<0$, transfer measure from $X_1^*$ to the tower.  
Choose disjoint sets $I_1^*(0),I_1^*(1),\ldots ,I_1^*(h_1-1)$ 
contained in $X_1^*$ such that $\mu (I_1^*(i))=d/h_1$. 
Denote $I_1^*=I_1^*(0)$. 
Begin by defining $\mu$ measure preserving map $\alpha_1$ such 
that $I_1^*(i+1)=\alpha_1 (I_1^*(i))$ for $i=0,1,\ldots ,h_1-2$. 
In this case, let 
$X_1^{\prime}=X_1^*\setminus [\bigcup_{i=0}^{h_1-1}I_1^*(i)]$. 

\subsubsection{$S$ Rescaling}
The direction mass is transferred depends on the sign of $b$ above. 
If $d>0$, then $\mu (J^{\prime}_1) > \mu (I^{\prime}_1)$ and 
mass is transferred from the residual $Y_1^*$ to the 
$S_1$-tower. 
Choose disjoint sets $J_1^*(0),J_1^*(1),\ldots ,J_1^*(h_1-1)$ 
contained in $Y_1^*$ such that $\mu (J_1^*(i))=d/h_1$. 
Denote $J_1^*=J_1^*(0)$. 
Begin by defining $\mu$ measure preserving map $\beta_1$ such 
that $J_1^*(i+1)=\beta_1 (J_1^*(i))$ for $i=0,1,\ldots ,h_1-2$. 
In this case, let 
$Y_1^{\prime}=Y_1^*\setminus [\bigcup_{i=0}^{h_1-1}J_1^*(i)]$. 
If $d=0$, set $Y_1^{\prime}=Y_1^*$. 
If $d<0$, transfer measure from the $S_1$-tower to the residual $Y_1^*$. 
Define $J_1^*\subset J_1\setminus J^{\prime}_1$ such that 
$\mu (J_1^*)=d/h_1$. Let 
$Y_1^{\prime}=Y_1^*\cup (\bigcup_{i=0}^{h_1-1}S_1^i(J_1^*))$. 


Note, if $d\neq 0$, 
then both $\epsilon_1$ and $\mu (X_1^*)$ may be chosen small enough 
(relative to $r_1$) to ensure the following solutions lead 
to well-defined sets and mappings. For subsequent stages, assume 
$\epsilon_n$ is chosen small enough to force well-defined 
rescaling parameters, transfer sets and mappings $R_n$, $S_n$. 

\subsection{Stage 2 Construction}
We have specified three cases: $d>0$, $d=0$ and $d<0$. The case $d=0$, can be handled along 
with the case $d>0$. This gives two essential cases. Note the case $d<0$ is analogous 
to the case $d>0$, with the roles of $R_1$ and $S_1$ reversed. However, 
due to a key distinction in the handling of the $R$-rescaling and the $S$-rescaling, 
it is important to clearly define $R_2$ and $S_2$ in both cases. 

\begin{case}[$d\geq 0$]
Define $\tau_1:X_1^{\prime}\to X_1^*$ as a measure preserving map between 
normalized spaces 
$(X_1^{\prime},\mathbb{B} \cap X_1^{\prime},\frac{\mu}{\mu (X_1^{\prime})})$ and 
$(X_1^*,\mathbb{B} \cap X_1^*,\frac{\mu}{\mu (X_1^*)})$. 
Extend $\tau_1$ to the new tower base, 
$$\tau_1:[I_1\setminus I_1^{\prime}]\cup [J^{\prime}_1\setminus I_1^*]\to I_1$$
such that $\tau_1$ preserves normalized measure between 
$$\frac{\mu}{\mu ([I_1\setminus I_1^{\prime}]\cup [J^{\prime}_1\setminus I_1^*])}\mbox{ and }\frac{\mu}{\mu (I_1)}.$$
Define $\tau_1$ on the remainder of the tower consistently 
such that 
\begin{eqnarray*} 
\tau_1(x)= 
\left\{\begin{array}{ll}
R_1^{i}\circ \tau_1 \circ R_1^{-i}(x) & \mbox{if $x\in R_1^i(I_1\setminus I_1^{\prime})$ for $0\leq i<h_1$} \\ 
R_1^{i}\circ \tau_1 \circ S_1^{-i}(x) & \mbox{if $x\in S_1^{i}(J^{\prime}_1\setminus I_1^*)$ for $0\leq i<h_1$}
\end{array}
\right.
\end{eqnarray*}
Define 
$R_2:X_2\to X_2$ as $R_2=\tau_1^{-1}\circ R_1\circ \tau_1$. Note 
\begin{eqnarray*} 
R_2(x)= 
\left\{\begin{array}{ll}
S_1(x) & \mbox{if $x\in S_1^{i}(J^{\prime}_1\setminus I_1^*)$ for $0\leq i<h_1-1$} \\ 
R_1(x) & \mbox{if $x\in R_1^{i}(I_1\setminus I^{\prime}_1)$ for $0\leq i<h_1-1$} 
\end{array}
\right.
\end{eqnarray*}
Clearly, $R_2$ is isomorphic to $R_1$ and $R$. 

Define $\psi_1:Y_1^{\prime}\to Y_1^*$ as a measure preserving map between 
normalized spaces 
$(Y_1^{\prime},\mathbb{B} \cap Y_1^{\prime},\frac{\mu}{\mu (Y_1^{\prime})})$ and 
$(Y_1^*,\mathbb{B} \cap Y_1^*,\frac{\mu}{\mu (Y_1^*)})$. 
Extend $\psi_1$ to the new tower base, 
$$\psi_1:[J_1\setminus J_1^{\prime}]\cup J_1^*\cup I^{\prime}_1\to J_1$$
such that $\psi_1$ preserves normalized measure between 
$$\frac{\mu}{\mu ([J_1\setminus J_1^{\prime}]\cup J_1^*\cup I^{\prime}_1)}\mbox{ and }\frac{\mu}{\mu (J_1)}.$$
Define $\psi_1$ on the remainder of the tower consistently 
such that 
\begin{eqnarray*} 
\psi_1(x)= 
\left\{\begin{array}{ll}
S_1^{i}\circ \psi_1 \circ S_1^{-i}(x) & \mbox{if $x\in S_1^i(J_1\setminus J_1^{\prime})$ for $0\leq i<h_1$} \\ 
S_1^{i}\circ \psi_1 \circ R_1^{-i}(x) & \mbox{if $x\in R_1^{i}(I^{\prime}_1)$ for $0\leq i<h_1$} \\ 
\beta_1^i\circ \psi_1 \circ \beta_1^{-i}(x) & \mbox{if $x\in J_1^*(i)$ for $0\leq i<h_1$}
\end{array}
\right.
\end{eqnarray*}
In this case, define $S_2:Y_2\to Y_2$ such that 
$S_2=\psi_1^{-1}\circ S_1\circ \psi_1$. Note 
\begin{eqnarray*} 
S_2(x)= 
\left\{\begin{array}{ll}
R_1(x) & \mbox{if $x\in R_1^{i}I^{\prime}_1$ for $0\leq i<h_1-1$} \\ 
S_1(x) & \mbox{if $x\in S_1^{i}(J_1\setminus J^{\prime}_1)$ for $0\leq i<h_1-1$} \\
\beta_1(x) & \mbox{if $x\in J_1^*(i)$ for $0\leq i<h_1-1$} \\ 
\psi_1^{-1}\circ S_1\circ \psi_1(x) & \mbox{if $x\in Y_1^{\prime}\cup S_1^{h_1-1}(J_1\setminus J_1^{\prime})\cup R_1^{h_1-1}I^{\prime}_1\cup \beta_1^{h_1-1}J_1^*$} 
\end{array}
\right.
\end{eqnarray*}
and $S_2$ is isomorphic to $S_1$ and $S$. 
\end{case}

\begin{case}[$d<0$]  
Define $\tau_1:X_1^{\prime}\to X_1^*$ as a measure preserving map between 
normalized spaces 
$(X_1^{\prime},\mathbb{B} \cap X_1^{\prime},\frac{\mu}{\mu (X_1^{\prime})})$ 
and $(X_1^*,\mathbb{B} \cap X_1^*,\frac{\mu}{\mu (X_1^*)})$. 
Extend $\tau_1$ to the new tower base, 
$$\tau_1:[I_1\setminus I_1^{\prime}]\cup I_1^*\cup J^{\prime}_1\to I_1$$
such that $\tau_1$ preserves normalized measure between 
$$\frac{\mu}{\mu ([I_1\setminus I_1^{\prime}]\cup I_1^*\cup J^{\prime}_1)}\mbox{ and }\frac{\mu}{\mu (I_1)}.$$
Define $\tau_1$ on the remainder of the tower consistently 
such that 
\begin{eqnarray*} 
\tau_1(x)= 
\left\{\begin{array}{ll}
R_1^{i}\circ \tau_1 \circ R_1^{-i}(x) & \mbox{if $x\in R_1^i(I_1\setminus I_1^{\prime})$ for $0\leq i<h_1$} \\ 
R_1^{i}\circ \tau_1 \circ S_1^{-i}(x) & \mbox{if $x\in S_1^{i}(J^{\prime}_1)$ for $0\leq i<h_1$} \\ 
\alpha_1^i\circ \tau_1 \circ \alpha_1^{-i}(x) & \mbox{if $x\in I_1^*(i)$ for $0\leq i<h_1$}
\end{array}
\right.
\end{eqnarray*}
In this case, define $R_2:X_2\to X_2$ such that 
\begin{eqnarray*} 
R_2(x)= 
\left\{\begin{array}{ll}
S_1(x) & \mbox{if $x\in S_1^{i}J^{\prime}_1$ for $0\leq i<h_1-1$} \\ 
R_1(x) & \mbox{if $x\in R_1^{i}(I_1\setminus I^{\prime}_1)$ for $0\leq i<h_1-1$} \\
\alpha_1(x) & \mbox{if $x\in I_1^*(i)$ for $0\leq i<h_1-1$} \\ 
\tau_1^{-1}\circ R_1\circ \tau_1(x) & \mbox{if $x\in X_1^{\prime}\cup R_1^{h_1-1}(I_1\setminus I_1^{\prime})\cup S_1^{h_1-1}J^{\prime}_1\cup \alpha_1^{h_1-1}I_1^*$} 
\end{array}
\right.
\end{eqnarray*}
Clearly, $R_2$ is isomorphic to $R_1$ and $R$. 

Define $\psi_1:Y_1^{\prime}\to Y_1^*$ as a measure preserving map between 
normalized spaces 
$(Y_1^{\prime},\mathbb{B} \cap Y_1^{\prime},\frac{\mu}{\mu (Y_1^{\prime})})$ 
and $(Y_1^*,\mathbb{B} \cap Y_1^*,\frac{\mu}{\mu (Y_1^*)})$. 
Extend $\psi_1$ to the new tower base, 
$$\psi_1:[J_1\setminus (J_1^{\prime}\cup J_1^*)]\cup I^{\prime}_1\to J_1$$
such that $\psi_1$ preserves normalized measure between 
$$\frac{\mu}{\mu ([J_1\setminus (J_1^{\prime}\cup J_1^*)]\cup I^{\prime}_1)}\mbox{ and }\frac{\mu}{\mu (J_1)}.$$
Define $\psi_1$ on the remainder of the tower consistently 
such that 
\begin{eqnarray*} 
\psi_1(x)= 
\left\{\begin{array}{ll}
S_1^{i}\circ \psi_1 \circ S_1^{-i}(x) & \mbox{if $x\in S_1^i(J_1\setminus [J_1^{\prime}\cup J_1^*])$ for $0\leq i<h_1$} \\ 
S_1^{i}\circ \psi_1 \circ R_1^{-i}(x) & \mbox{if $x\in R_1^{i}(I^{\prime}_1)$ for $0\leq i<h_1$}
\end{array}
\right.
\end{eqnarray*}
Define $S_2:Y_2\to Y_2$ such that $S_2=\psi_1^{-1}\circ S_1\circ \psi_1$. Note 
\begin{eqnarray*} 
S_2(x)= 
\left\{\begin{array}{ll}
R_1(x) & \mbox{if $x\in R_1^{i}(I^{\prime}_1)$ for $0\leq i<h_1-1$} \\ 
S_1(x) & \mbox{if $x\in S_1^{i}(J_1\setminus [J^{\prime}_1\cup J_1^*])$ for $0\leq i<h_1-1$}
\end{array}
\right.
\end{eqnarray*}
Transformation $S_2$ is isomorphic to $S_1$ and $S$. 
\end{case}
Define $T_2$ as 
\begin{eqnarray*} 
T_2(x)= 
\left\{\begin{array}{ll}
R_2(x) & \mbox{if $x\in X_2$} \\ 
S_2(x) & \mbox{if $x\in Y_2$}
\end{array}
\right.
\end{eqnarray*}
Clearly, neither $T_1$ nor $T_2$ are ergodic. For $T_1$, $X_1$ and $Y_1$ 
are ergodic components, and $X_2$, $Y_2$ are ergodic components for 
$T_2$. 
\subsection{General Multiplexing Operation}
For $n\geq 1$, suppose that $R_n$ and $S_n$ have been defined on $X_n$ and $Y_n$ respectively. Construct Rohklin towers of height $h_n$ for each $R_n$ and $S_n$, and such that $I_n$ is the base of the $R_n$ tower, $J_n$ is the base of the $S_n$ tower, and $\mu (\bigcup_{i=0}^{h_n-1}R_n^iI_n) + \mu (\bigcup_{i=0}^{h_n-1}S_n^iJ_n)>1-\epsilon_n$. Let $I^{\prime}_n \subset I_n$ be such that $\mu (I^{\prime}_n) = r_n \mu (I_n)$. Similarly, suppose $J^{\prime}_n \subset J_n$ such that $\mu (J^{\prime}_n) = s_n \mu (J_n)$. 

We define $R_{n+1}$ and $S_{n+1}$ by switching the subcolumns 
$$\{I^{\prime}_n, R_n(I^{\prime}_n),R_n^2(I^{\prime}_n),\ldots ,R_n^{h_n-1}(I^{\prime}_n)\}$$ and 
$$\{J^{\prime}_n, S_n(J^{\prime}_n),S_n^2(J^{\prime}_n),\ldots ,S_n^{h_n-1}(J^{\prime}_n)\}.$$ Let 
\begin{eqnarray*}
X_{n+1}&=&[\bigcup_{i=0}^{h_n-1}R_n^i(I_n\setminus I_n^{\prime})]
\cup [\bigcup_{i=0}^{h_n-1}S_n^iJ_n^{\prime}] 
\cup [X_n\setminus \bigcup_{i=0}^{h_n-1}R_n^iI_n] \\ 
Y_{n+1}&=&[\bigcup_{i=0}^{h_n-1}S_n^i(J_n\setminus J_n^{\prime})]
\cup [\bigcup_{i=0}^{h_n-1}R_n^iI_n^{\prime}] 
\cup [Y_n\setminus \bigcup_{i=0}^{h_n-1}S_n^iJ_n].
\end{eqnarray*}
As in the initial case, it may be necessary to transfer measure 
between each column and its respective residual. 
We can follow the same algorithm as above, and define 
maps $\tau_n, \alpha_n, \psi_n$ and $\beta_n$. 
Thus, we get the following definitions: 
\begin{case}[$d\geq 0$]
\begin{eqnarray*}
\tau_n(x)=& 
\left
\{\begin{array}{ll}
R_n^{i}\circ \tau_n \circ R_n^{-i}(x) & \mbox{if $x\in R_n^i(I_n\setminus I_n^{\prime})$ for $0\leq i<h_n$} \\ 
R_n^{i}\circ \tau_n \circ S_n^{-i}(x) & \mbox{if $x\in S_n^{i}(J^{\prime}_n\setminus I_1^*)$ for $0\leq i<h_n$}
\end{array}
\right.
\end{eqnarray*}
\begin{eqnarray*} 
R_{n+1}(x)=& 
\left\{\begin{array}{ll}
S_n(x) & \mbox{if $x\in S_n^{i}(J^{\prime}_n\setminus I_n^*)$ for $0\leq i<h_n-1$} \\ 
R_n(x) & \mbox{if $x\in R_n^{i}(I_n\setminus I^{\prime}_n)$ for $0\leq i<h_n-1$} \\
\tau_n^{-1}\circ R_n\circ \tau_n(x) & \mbox{if $x\in X_n^{\prime}\cup R_n^{h_n-1}(I_n\setminus I_n^{\prime})\cup S_n^{h_n-1}(J^{\prime}_n\setminus I_n^*)$} 
\end{array}
\right.
\end{eqnarray*}
and $R_{n+1}=\tau_n^{-1}\circ R_n\circ \tau_n$.
\begin{eqnarray*} 
\psi_n(x)=& 
\left\{\begin{array}{ll}
S_n^{i}\circ \psi_n \circ S_n^{-i}(x) & \mbox{if $x\in S_n^i(J_n\setminus J_n^{\prime})$ for $0\leq i<h_n$} \\ 
S_n^{i}\circ \psi_n \circ R_n^{-i}(x) & \mbox{if $x\in R_n^{i}(I^{\prime}_n)$ for $0\leq i<h_n$} \\ 
\beta_n^i\circ \psi_n \circ \beta_n^{-i}(x) & \mbox{if $x\in J_n^*(i)$ for $0\leq i<h_n$}
\end{array}
\right.
\end{eqnarray*}
\begin{eqnarray*} 
S_{n+1}(x)=& 
\left\{\begin{array}{ll}
R_n(x) & \mbox{if $x\in R_n^{i}I^{\prime}_n$ for $0\leq i<h_n-1$} \\ 
S_n(x) & \mbox{if $x\in S_n^{i}(J_n\setminus J^{\prime}_n)$ for $0\leq i<h_n-1$} \\
\beta_n(x) & \mbox{if $x\in J_n^*(i)$ for $0\leq i<h_n-1$} \\ 
\psi_n^{-1}\circ S_n\circ \psi_n(x) & \mbox{if $x\in Y_n^{\prime}\cup S_n^{h_n-1}(J_n\setminus J_n^{\prime})\cup R_n^{h_n-1}I^{\prime}_n\cup \beta_n^{h_n-1}J_n^*$} 
\end{array}
\right.
\end{eqnarray*}
and $S_{n+1}=\psi_n^{-1}\circ S_n\circ \psi_n$.
\end{case}

\begin{case}[$d<0$]
\begin{eqnarray*} 
\tau_n(x)=& 
\left\{\begin{array}{ll}
R_n^{i}\circ \tau_n \circ R_n^{-i}(x) & \mbox{if $x\in R_n^i(I_n\setminus I_n^{\prime})$ for $0\leq i<h_n$} \\ 
R_n^{i}\circ \tau_n \circ S_n^{-i}(x) & \mbox{if $x\in S_n^{i}(J^{\prime}_n)$ for $0\leq i<h_n$} \\ 
\alpha_n^i\circ \tau_n \circ \alpha_n^{-i}(x) & \mbox{if $x\in I_n^*(i)$ for $0\leq i<h_n$}
\end{array}
\right.
\end{eqnarray*}
\begin{eqnarray*} 
R_{n+1}(x)=& 
\left\{\begin{array}{ll}
S_n(x) & \mbox{if $x\in S_n^{i}J^{\prime}_n$ for $0\leq i<h_n-1$} \\ 
R_n(x) & \mbox{if $x\in R_n^{i}(I_n\setminus I^{\prime}_n)$ for $0\leq i<h_n-1$} \\
\alpha_n(x) & \mbox{if $x\in I_n^*(i)$ for $0\leq i<h_n-1$} \\ 
\tau_n^{-1}\circ R_n\circ \tau_n(x) & \mbox{if $x\in X_n^{\prime}\cup R_n^{h_n-1}(I_n\setminus I_n^{\prime})\cup S_n^{h_n-1}J^{\prime}_n\cup \alpha_n^{h_n-1}I_n^*$} 
\end{array}
\right.
\end{eqnarray*}
and $R_{n+1}=\tau_n^{-1}\circ R_n\circ \tau_n$.
\begin{eqnarray*} 
\psi_n(x)=& 
\left\{\begin{array}{ll}
S_n^{i}\circ \psi_n \circ S_n^{-i}(x) & \mbox{if $x\in S_n^i(J_n\setminus [J_n^{\prime}\cup J_n^*])$ for $0\leq i<h_n$} \\ 
S_n^{i}\circ \psi_n \circ R_n^{-i}(x) & \mbox{if $x\in R_n^{i}(I^{\prime}_n)$ for $0\leq i<h_n$}
\end{array}
\right.
\end{eqnarray*}
\begin{eqnarray*} 
S_{n+1}(x)=& 
\left\{\begin{array}{ll}
R_n(x) & \mbox{if $x\in R_n^{i}(I^{\prime}_n)$ for $0\leq i<h_n-1$} \\ 
S_n(x) & \mbox{if $x\in S_n^{i}(J_n\setminus [J^{\prime}_n\cup J_n^*])$ for $0\leq i<h_n-1$} \\
\psi_n^{-1}\circ S_n\circ \psi_n(x) & \mbox{if $x\in Y_n^{\prime}\cup S_n^{h_n-1}(J_n\setminus [J_n^{\prime}\cup J_n^*])\cup R_n^{h_n-1}(I^{\prime}_n)$} 
\end{array}
\right.
\end{eqnarray*}
and $S_{n+1}=\psi_n^{-1}\circ S_n\circ \psi_n$.
\end{case}
\subsection{The Limiting Transformation}
Define the transformation 
$T_{n+1}:X_{n+1}\cup Y_{n+1}\to X_{n+1}\cup Y_{n+1}$ such that 
\begin{eqnarray*} 
T_{n+1}(x)= 
\left\{\begin{array}{ll}
R_{n+1}(x) & \mbox{if $x\in X_{n+1}$} \\ 
S_{n+1}(x) & \mbox{if $x\in Y_{n+1}$} 
\end{array}
\right.
\end{eqnarray*}
The set where $T_{n+1}\neq T_n$ is determined by the top levels 
of the Rokhlin towers, the residuals and the transfer sets. 
Note the transfer set has measure $d$. Since this set is used to adjust 
the size of the residuals between stages, it can be bounded below 
a constant multiple of $\epsilon_n$. Thus, there is a fixed constant $\kappa$, 
independent of $n$, such that 
$T_{n+1}(x)=T_n(x)$ except for $x$ in a set of measure less than 
$\kappa (\epsilon_n + {1}/{h_n})$. Since $\sum_{n=1}^{\infty} (\epsilon_n + {1}/{h_n}) < \infty$, 
$T(x) = \lim_{n\to \infty}T_n(x)$ exists almost everywhere, 
and preserves normalized Lebesgue measure. 
Without loss of generality, we may assume $\kappa$ and $h_n$ 
are chosen such that if 
$$E_n=\{x\in X| T_{n+1}(x)\neq T_n(x)\}$$
then $\mu (E_n) < \kappa \epsilon_n$ for $n\in \natural$. 
In the following section, additional structure and conditions 
are implemented to ensure that 
$T$ inherits properties from $R$ and $S$, and is also ergodic. 

For the remainder of this paper, 
assume the parameters are chosen such that 
\begin{enumerate}
\item $\lim_{n\to \infty}s_n= 0$;
\item $\sum_{n=1}^{\infty}r_n=\sum_{n=1}^{\infty}s_n=\infty$; 
\item $\lim_{n\to \infty}\mu (X_n)=0$;
\item $\sum_{n=1}^{\infty}\epsilon_n<\infty$.
\end{enumerate}

\subsection{Isomorphism Chain Consistency}
Suppose $S$ is a strong mixing transformation on $(Y,\mathcal{B},\mu)$. 
We will use the multiplexing procedure defined in the previous section 
to produce a "slow" mixing transformation $T$.  
Let $\mu_n$ be normalized Lebesgue probability measure on $Y_n$. 
i.e. $\mu_n = {\mu} / {\mu(Y_n)}$. 

For $n\in \natural$, let $P_n$ be a refining sequence of finite partitions 
which generates the sigma algebra. By refining $P_n$ further if necessary, 
assume $X_n,Y_n,X_n^*,Y_n^*\in P_n$. 
Also, assume 
$R_n^i(I^{\prime}_n), R_n^i(I_n\setminus I^{\prime}_n), S_n^i(J^{\prime}_n), 
S_n^i(J_n\setminus J^{\prime}_n)$ are elements of $P_n$ 
for $0\leq i<h_n$. 
Finally, assume for $0\leq i<h_n-1$, if $p\in P_n$ and 
$p\subset R_n^i(I_n)$ then $R_n(p)\in P_n$. 
Likewise, assume for $0\leq i<h_n-1$, if $p\in P_n$ and 
$p\subset S^i(J_n)$ then $S_n(p)\in P_n$. 
Previously, we required that $\psi_n$ map certain finite orbits from 
the $S_n$ and $R_n$ towers to a corresponding orbit in the $S_{n+1}$ tower. 
In this section, further regularity is imposed on $\psi_n$ relative to $P_n$ 
to ensure dynamical properties of $S_n$ are inherited by $S_{n+1}$. 

Let $P_n^{\prime}=
\{p\in P_n| p\subset \bigcup_{i=0}^{h_n-1}S_n^i(J_n\setminus J_n^{\prime})\}$. 
For each of the following three cases, impose the corresponding 
restriction on $\psi_n$:
\begin{enumerate}
\item for $d=0$ and $p\in P_n^{\prime}$, $\psi_n$ is the identity map (i.e. $\psi_n(p)=p$);
\item for $d>0$ and $p\in P_n^{\prime}$, $\psi_n(p)\subset p$;
\item for $d<0$ and $p\in P_n^{\prime}$, $p\subset \psi_n(p)$.
\end{enumerate}
This can be accomplished by uniformly distributing the appropriate mass 
from the sets $S_n^i(J_n^*)$ using $\psi_n$. 
Note that $\psi_n$ either preserves Lebesgue measure in the case $d=0$, 
or $\psi_n$ contracts sets relative to Lebesgue measure in the case $d>0$, 
or it inflates measure in the case $d<0$. In all three cases, 
for $p\in P_n^{\prime}$,
$$
\frac{\mu (p)}{\mu (\psi_n (p))}=\frac{\mu (Y_{n+1})}{\mu (Y_n)}.
$$
It is straightforward to verify for any set $A$ measurable relative to $P_n^{\prime}$, 
$$
\mu (A\triangle \psi_nA)<|\frac{\mu (Y_{n+1})}{\mu (Y_n)} - 1|.
$$
The properties of $\psi_n$ allow approximation of $S_{n+1}$ by $S_n$ indefinitely over time. 
This is needed to establish mixing for the limiting transformation $T$. 

Since each $S_n$ is strongly mixing on $Y_n$, 
then for all $A,B\in P_n^{\prime}$, 
\[
\lim_{i \to \infty} \mu_n(A\cap S_n^iB) = \mu_n(A)\mu_n(B) . 
\]

Prior to establishing strong mixing, we prove 
a lemma which is part of a similar lemma shown in \cite{Towerplex1}. 
For $p\in P_n^{\prime}$,
$$
\frac{\mu (p)}{\mu (\psi_n (p))}=\frac{\mu (Y_{n+1})}{\mu (Y_n)}.
$$
It is straightforward to verify for any set $A$ measurable relative to $P_n^{\prime}$, 
\begin{align*}
\mu (A\triangle \psi_nA) &= \mu (A) - \mu (\psi_n(A)) \\ 
&\leq \mu (\psi_n A) [ \frac{\mu (Y_{n+1})}{\mu (Y_n)} - 1 ] 
= \frac{\mu (\psi_n A)}{\mu (Y_n)} [ \mu (Y_{n+1}) - \mu (Y_n) ] . 
\end{align*}
and for any measurable set $C\subset Y_n$, 
$$
| \mu (\psi_n^{-1}C) - \mu(C) | <  |\frac{\mu (Y_{n+1})}{\mu (Y_n)} - 1|.
$$
These two properties are used in the following lemma to show $S_{n+1}$ inherits dynamical 
properties from $S_n$ indefinitely over time.  
Let $Q_n = \{ \psi_n (p) : p \in P_n^{\prime} \}$.  
\begin{lemma}
\label{rescalinglem}
Suppose $\delta >0$ and $n\in \natural$ is chosen such that 
\begin{eqnarray*}
\epsilon_n + \mu (X_n) < \frac{\delta}{6}.
\end{eqnarray*}
Then for $A,B\in Q_n$ and $i\in \natural$, 
\[
|\mu (S_{n+1}^iA\cap B)-\mu (A)\mu (B) | 
< |\mu (S_{n}^iA\cap B)-\mu (A)\mu (B) | + \delta . 
\]
\end{lemma}
\begin{proof} 
For $A,B\in Q_n$, let $A^{\prime}= \psi_n^{-1} A$ and 
$B^{\prime}=\psi_n^{-1} B$. 
Thus, 
$\mu(A^{\prime}\triangle A) = \mu(\psi_n^{-1}(A\setminus \psi_nA)) < {\delta} / {6}$ 
and $\mu (B^{\prime}\triangle B)<\frac{\delta}{6}$. 
By applying the triangle inequality several times, we get the following approximation: 
\begin{eqnarray*} 
| \mu (S_{n+1}^iA\cap B) &-& \mu(S_n^i A \cap B) | \\ 
&\leq& | \mu (S_{n+1}^iA^{\prime}\cap B^{\prime}) - \mu(S_n^i A \cap B) | 
+ \frac{\delta}{3} \\ 
&=& \mu (\psi_n^{-1}S_{n}^i\psi_nA^{\prime}\cap B^{\prime}) - \mu(S_n^i A \cap B) | + \frac{\delta}{3} \\ 
&=& \mu (\psi_n^{-1}(S_{n}^i\psi_nA^{\prime} \cap \psi_n B^{\prime})) - \mu(S_n^i A \cap B) | + \frac{\delta}{3} \\ 
&=& | \mu( \psi_n^{-1}(S_n^i A \cap B)) - \mu(S_n^i A \cap B) | + \frac{\delta}{3} \\ 
&<& \frac{\delta}{2} . 
\end{eqnarray*}
Similarly, 
$$
| \mu (S_{n+1}^iF\cap F) - \mu(S_n^i F \cap F) | < \frac{\delta}{2} . 
$$
Therefore, 
$$|\mu (S_{n+1}^iA\cap B)-\mu (A)\mu (B) | < |\mu (S_{n}^iA\cap B)-\mu (A)\mu (B) | + \delta.$$
\end{proof}

\section{Towerplexes with singular spectrum}
If $\Phi$ is the space of ergodic measure preserving transformations on a separable 
probability space, then the tower multiplexing operation defines a mapping 
\[
\mathcal{M}: \Phi \times \Phi \to \Phi .
\]
The mapping also depends on a collection of parameters $\mathcal{P}$. 
Thus, we may write $T = \mathcal{M}(R,S,\mathcal{P})$ to represent the multiplexed 
transformation $T$ produced from transformations $R$ and $S$. 
In \cite{Towerplex1}, the transformation $R$ is ergodic and rigid, and $S$ is weak mixing. 
In particular, $S$ is set to the Chacon3 transformation, and 
the parameters are defined such that $S$ is a "dissipating" component.  
Given $R$ ergodic with rigidity sequence $(\rho_n)_{n=1}^{\infty}$, 
it is shown there exists $\mathcal{P}$ such that 
$T = \mathcal{M}(R,S,\mathcal{P})$ is weak mixing with rigidity sequence $\rho_n$. 

In this paper, we use the tower multiplexing technique to produce a transformation 
$T=\mathcal{M}(S,R,\mathcal{P})$ with continuous singular spectrum. 
Again, the second component transformation will be a dissipating component. 
However, we flip the roles of $R$ and $S$, so $R$ (rigidity) is used in the second component. 
The parameter collection $\mathcal{P}$ includes sequences $r_n$ and $s_n$. As in \cite{Towerplex1}, 
associate $r_n$ with $R$ and $s_n$ with $S$. Associate $X_n$ with the second component 
transformation $R$, and $Y_n$ with $S$. Other parameters included in $\mathcal{P}$ 
are $\epsilon_n$ and $h_n$. 
We have the following main theorem. 
\begin{theorem}
\label{towerplex0}
Let $S$ be an invertible ergodic measure preserving transformation 
with weak limit $S^{m_n} \to S_0$ as $n\to \infty$. 
There exist a rigid weak mixing transformation $R$, and parameter $\mathcal{P}$ such that 
\[
T = \mathcal{M}(S,R,\mathcal{P}) 
\]
is weak mixing with singular spectrum, and $T^{m_n} \to S_0$.  
\end{theorem}
Prior to sketching a proof to the previous theorem, we will use the following result from \cite{Fay04}. 
\begin{proposition}[B. Fayad]
\label{slowcoal}
Let $(X,\mathcal{B},\mu,T)$ be an invertible ergodic measure preserving system. If for any complex 
nonzero, mean zero function $f\in L^2(X,\mu)$, there exists a measurable set $E\subset X$ 
with $\mu (E) > 0$, and a strictly increasing sequence $\ell_n$, such that for every $x\in E$, 
we have 
\[
\limsup_{n\to \infty} \frac{1}{n} | \sum_{i=0}^{n-1} f(T^{\ell_i}x)| > 0, 
\]
then the maximal spectral type of the unitary operator associated to 
\newline \noindent $(X,\mathcal{B},\mu,T)$ is singular. 
\end{proposition}

\noindent {\bf Proof of Theorem \ref{towerplex0}:} 
Let $S$ be an invertible ergodic measure preserving transformation with limit 
$S_0 = w^{*}-\lim_{n\to \infty} S^{m_n}$.  
We can define a rigid weak mixing transformation $R$ such that the dissipating component $R$ 
in the multiplexed transformation $T = \mathcal{M}( S,R, \mathcal{P} )$ 
will allow $T$ to satisfy the previous proposition. 
We still require that parameters $r_n$ and $s_n$ have the same properties as in \cite{Towerplex1} 
except with roles reversed. 
In particular, $r_n$ for $R$ satisfies $r_n = 1/2$, and $s_n = 1/{2(n+2)}$. 
This ensures that the base of $R_n$, $X_n$, satisfies 
$\lim_{n\to \infty} \mu(X_n) = 0$ and $\sum_{n=1}^{\infty} \mu(X_n)=\infty$ with the $X_n$ 
approximately independent. 
The same technique to establish the rigidity sequence in \cite{Towerplex1} can be used 
to establish weak convergence to $S_0$ along $m_n$. 
Ergodicity and weak mixing may be established in a similar manner as in \cite{Towerplex1}. 
Singular spectrum is established using the previous proposition and the fact that 
almost every point falls in a subset of $X_n$ infinitely often. The transformation $R$ is defined 
such that $\ell_i$ are "strong" rigid times, and rigid multiples of rigid times.  $\Box$ 
\begin{corollary}
\label{countable}
Suppose $S$ is an invertible strong mixing transformation. There exist an invertible rigid transformation $R$, 
and parameter $\mathcal{P}$ such that 
\[
T = \mathcal{M}(S,R,\mathcal{P}) 
\]
is strong mixing with singular spectrum. 
\end{corollary}

\begin{corollary}
\label{countable}
Suppose $S$ is an invertible ergodic measure preserving transformation, and $(m_n)_{n=1}^{\infty}$ 
is a sequence such that the weak closure of $\{ S^{m_n}: n\in \natural \}$ contains a countable set of limit points 
$\{ S_k : k\in \natural \}$. Then there exists a weak mixing transformation $T$ with singular spectrum 
such that the weak closure of $\{ T^{m_n}: n\in \natural \}$ contains the same countable set of limit points 
$\{ S_k : k\in \natural \}$. 
\end{corollary}

\subsection{Multiple mixing towerplexes}
In \cite{Host91}, it is shown that any mixing transformation with singular spectrum 
is mixing of all orders.  This implies the following corollary. 
\begin{corollary}
\label{all_orders}
Given any strong mixing transformation $S$, 
there exist a rigid transformation $R$, 
and parameter $\mathcal{P}$ such that 
\[
T = \mathcal{M}(S,R,\mathcal{P}) 
\]
is mixing of all orders. 
\end{corollary}

\noindent 
{\bf Question:} Given a mixing transformation $S$, 
is it possible to construct a rigid transformation $R$ and parameter $\mathcal{P}$ 
such that 
$T = \mathcal{M}(S,R,\mathcal{P})$ has the same higher order mixing 
properties as $S$ and has singular spectrum? 
This would be sufficient to prove that strong mixing 
implies mixing of all orders. 

\subsection*{Acknowledgements}
The author wishes to thank Vitaly Bergelson for pointing out the question 
in \cite{Berg96_2}.

\appendix
\section{Over-recurrent sets for mixing transformations}

\begin{figure}[h]
\caption{Over-recurrent sets}

\begin{tikzpicture}
{[line width=2pt]
{[black]
\draw (1+.5,0) -- (1+4.5,0) node[midway,below](base) {$B_1$=base, $A_1$=red};
\draw (1+.5,.4) -- (1+4.5,.4);
\draw (1+.5,.8) -- (1+4.5,.8);
\draw (1+.5,1.2) -- (1+4.5,1.2) node[near start, below](s1) {};
\node[black, left=of s1]{$m_1$};
\draw (1+.5,1.6) -- (1+4.5,1.6);
\draw (1+.5,2) -- (1+4.5,2);

\draw (1+.5,2.8) -- (1+4.5,2.8);
\draw (1+.5,3.2) -- (1+4.5,3.2);
\draw (1+.5,3.6) -- (1+4.5,3.6);
\draw (1+.5,4) -- (1+4.5,4) node[near start, below](s2) {};
\node[black, left=of s2]{$m_1$};
\draw (1+.5,4.4) -- (1+4.5,4.4);
\draw (1+.5,4.8) -- (1+4.5,4.8);

\draw (1+.5,5.6) -- (1+4.5,5.6);
\draw (1+.5,6) -- (1+4.5,6);
\draw (1+.5,6.4) -- (1+4.5,6.4);
\draw (1+.5,6.8) -- (1+4.5,6.8) node[near start, below](s3) {};
\node[black, left=of s3]{$m_1$};
\draw (1+.5,7.2) -- (1+4.5,7.2);
\draw (1+.5,7.6) -- (1+4.5,7.6);

\draw (1+.5,8.4) -- (1+4.5,8.4);
\draw (1+.5,8.8) -- (1+4.5,8.8);
\draw (1+.5,9.2) -- (1+4.5,9.2);
\draw (1+.5,9.6) -- (1+4.5,9.6) node[near start, below](s4) {};
\node[black, left=of s4]{$m_1$};
\draw (1+.5,10) -- (1+4.5,10);
\draw (1+.5,10.4) -- (1+4.5,10.4);

\draw (1+.5,11.2) -- (1+4.5,11.2);
\draw (1+.5,11.6) -- (1+4.5,11.6);
\draw (1+.5,12) -- (1+4.5,12);
\draw (1+.5,12.4) -- (1+4.5,12.4) node[near start, below](s5) {};
\node[black, left=of s5]{$m_1$};
\draw (1+.5,12.8) -- (1+4.5,12.8);
\draw (1+.5,13.2) -- (1+4.5,13.2);

\draw (1+.5,14) -- (1+4.5,14);
\draw (1+.5,14.4) -- (1+4.5,14.4);
\draw (1+.5,14.8) -- (1+4.5,14.8);
\draw (1+.5,15.2) -- (1+4.5,15.2) node[near start, below](s6) {};
\node[black, left=of s6]{$m_1$};
\draw (1+.5,15.6) -- (1+4.5,15.6);
\draw (1+.5,16) -- (1+4.5,16);
}
}
{[line width=4pt]
{[red]
\draw (1+.5,0) -- (1+2,0);
\draw (1+2,1.2) -- (1+4.5,1.2);

\draw (1+2.25,2.8) -- (1+2.75,2.8);
\draw (1+.5,3.2) -- (1+1.5,3.2);
\draw (1+2.75,3.6) -- (1+4.5,3.6);
\draw (1+1.5,4.4) -- (1+2.25,4.4);

\draw (1+.5,6.8) -- (1+4.4,6.8);

\draw (1+2,8.4) -- (1+3,8.4);
\draw (1+3.1,9.2) -- (1+4.5,9.2);
\draw (1+.6,10.4) -- (1+1.9,10.4);

\draw (1+2.25,12.8) -- (1+2.75,12.8);
\draw (1+.5,11.6) -- (1+1.5,11.6);
\draw (1+2.75,13.2) -- (1+4.5,13.2);
\draw (1+1.5,11.2) -- (1+2.25,11.2);

\draw (1+2.25,12.8) -- (1+2.75,12.8);
\draw (1+.5,11.6) -- (1+1.5,11.6);
\draw (1+2.75,13.2) -- (1+4.5,13.2);
\draw (1+1.5,11.2) -- (1+2.25,11.2);

\draw (1+2,15.2) -- (1+3,15.2);
\draw (1+2.9,16) -- (1+4.5,16);
\draw (1+.5,14.4) -- (1+2.1,14.4);
}
}
\shade [ball color=blue] (1+2.25,16.9) circle (.16) node[](residual){};
\draw[->] (1+3.0,16.9)--(1+2.5,16.9) node[blue](myarrow1){};
\node[purple] at (1+3.25,16.8) {$\epsilon_1$};
\node[black,left=of residual]{$X$};
{[line width=2pt]
{[black]
\draw (8+.5,0) -- (8+4.5,0) node[midway,below] {$I_1$=orange, $A_2$=blue, $A_1$=red};
\draw (8+.5,.4) -- (8+4.5,.4);
\draw (8+.5,.8) -- (8+4.5,.8);
\draw (8+.5,1.2) -- (8+4.5,1.2);
\draw (8+.5,1.6) -- (8+4.5,1.6);
\draw (8+.5,2) -- (8+4.5,2);

\draw (8+.5,2.8) -- (8+4.5,2.8);
\draw (8+.5,3.2) -- (8+4.5,3.2);
\draw (8+.5,3.6) -- (8+4.5,3.6);
\draw (8+.5,4) -- (8+4.5,4);
\draw (8+.5,4.4) -- (8+4.5,4.4);
\draw (8+.5,4.8) -- (8+4.5,4.8);

\draw (8+.5,5.6) -- (8+4.5,5.6);
\draw (8+.5,6) -- (8+4.5,6);
\draw (8+.5,6.4) -- (8+4.5,6.4);
\draw (8+.5,6.8) -- (8+4.5,6.8);
\draw (8+.5,7.2) -- (8+4.5,7.2);
\draw (8+.5,7.6) -- (8+4.5,7.6);

\draw (8+.5,8.4) -- (8+4.5,8.4);
\draw (8+.5,8.8) -- (8+4.5,8.8);
\draw (8+.5,9.2) -- (8+4.5,9.2);
\draw (8+.5,9.6) -- (8+4.5,9.6);
\draw (8+.5,10) -- (8+4.5,10);
\draw (8+.5,10.4) -- (8+4.5,10.4);

\draw (8+.5,11.2) -- (8+4.5,11.2);
\draw (8+.5,11.6) -- (8+4.5,11.6);
\draw (8+.5,12) -- (8+4.5,12);
\draw (8+.5,12.4) -- (8+4.5,12.4);
\draw (8+.5,12.8) -- (8+4.5,12.8);
\draw (8+.5,13.2) -- (8+4.5,13.2);

\draw (8+.5,14) -- (8+4.5,14);
\draw (8+.5,14.4) -- (8+4.5,14.4);
\draw (8+.5,14.8) -- (8+4.5,14.8);
\draw (8+.5,15.2) -- (8+4.5,15.2);
\draw (8+.5,15.6) -- (8+4.5,15.6);
\draw (8+.5,16) -- (8+4.5,16);
}
}
{[line width=4pt]
{[red]
\draw (8+.5,0) -- (8+2,0);
\draw (8+2,1.2) -- (8+4.5,1.2);

\draw[orange] (8+.5,0) -- (8+1.25,0);

\draw[blue] (8+.5,.4) -- (8+1.25,.4);
\draw[blue] (8+.5,.8) -- (8+1.25,.8);
\draw[blue] (8+.5,1.2) -- (8+1.25,1.2);
\draw[blue] (8+.5,1.6) -- (8+1.25,1.6);
\draw[blue] (8+.5,2) -- (8+1.25,2);

\draw (8+2.25,2.8) -- (8+2.75,2.8);
\draw (8+.5,3.2) -- (8+1.5,3.2);
\draw (8+2.75,3.6) -- (8+4.5,3.6);
\draw (8+1.5,4.4) -- (8+2.25,4.4);

\draw[blue] (8+.5,2.8) -- (8+1.25,2.8);
\draw[blue] (8+.5,3.6) -- (8+1.25,3.6);
\draw[blue] (8+.5,4) -- (8+1.25,4);
\draw[blue] (8+.5,4.4) -- (8+1.25,4.4);
\draw[blue] (8+.5,4.8) -- (8+1.25,4.8);

\draw (8+.5,6.8) -- (8+4.4,6.8);

\draw[blue] (8+.5,7.6) -- (8+1.25,7.6);
\draw[blue] (8+.5,5.6) -- (8+1.25,5.6);
\draw[blue] (8+.5,6) -- (8+1.25,6);
\draw[blue] (8+.5,6.4) -- (8+1.25,6.4);
\draw[blue] (8+.5,7.2) -- (8+1.25,7.2);

\draw (8+2,8.4) -- (8+3,8.4);
\draw (8+3.1,9.2) -- (8+4.5,9.2);
\draw (8+.7,10.4) -- (8+1.9,10.4);

\draw[blue] (8+.5,8.4) -- (8+1.25,8.4);
\draw[blue] (8+.5,8.8) -- (8+1.25,8.8);
\draw[blue] (8+.5,9.2) -- (8+1.25,9.2);
\draw[blue] (8+.5,9.6) -- (8+1.25,9.6);
\draw[blue] (8+.5,10) -- (8+1.25,10);
\draw[blue] (8+.5,10.4) -- (8+.7,10.4);

\draw (8+2.25,12.8) -- (8+2.75,12.8);
\draw (8+.5,11.6) -- (8+1.5,11.6);
\draw (8+2.75,13.2) -- (8+4.5,13.2);
\draw (8+1.5,11.2) -- (8+2.25,11.2);

\draw[blue] (8+.5,11.2) -- (8+1.25,11.2);
\draw[blue] (8+.5,12) -- (8+1.25,12);
\draw[blue] (8+.5,12.4) -- (8+1.25,12.4);
\draw[blue] (8+.5,12.8) -- (8+1.25,12.8);
\draw[blue] (8+.5,13.2) -- (8+1.25,13.2);

\draw (8+2,15.2) -- (8+3,15.2);
\draw (8+2.9,16) -- (8+4.5,16);
\draw (8+.5,14.4) -- (8+2.1,14.4);

\draw[blue] (8+.5,14) -- (8+1.25,14);
\draw[blue] (8+.5,14.8) -- (8+1.25,14.8);
\draw[blue] (8+.5,15.2) -- (8+1.25,15.2);
\draw[blue] (8+.5,15.6) -- (8+1.25,15.6);
\draw[blue] (8+.5,16) -- (8+1.25,16);
}
}
\shade [ball color=blue] (8+2.25,16.9) circle (.16) node[](residual){};
\draw[->] (8+3.0,16.9)--(8+2.5,16.9) node[blue](myarrow1){};
\node[purple] at (8+3.25,16.8) {$\epsilon_1$};
\node[black,left=of residual]{$X$};

\end{tikzpicture}
\label{fig:tower}
\end{figure}
For the example in Figure \ref{fig:tower}, $m_1 = 6$.  
The height is $m_1^2 = 36$. 
The set $A \supset A_1 \cup A_2$. For $i$ such that $0 \leq i < N_1 < m_1$, 
$T^i(A_2)$ will be contained in $A$ except for part of the top $m_1$-block. 
Since $m_1 > {1} / { ( \epsilon_1 a_2 ) }$, then the top $m$-block has measure 
less than $\epsilon_1 a_2$.  Thus, the proportion of $A_2$ that does not 
return to $A$ under $T^i$, $0 \leq i < N_1 < m_1$, is less than $\epsilon_1 \mu(A_2)$. 
A similar argument holds for $- N_1 < i \leq 0$.

\end{document}